\documentclass[10pt]{article} 
\usepackage{amsmath,amssymb,amsthm} 
\usepackage[OT2,T1]{fontenc}
\usepackage[utf8]{inputenc}
\usepackage{enumitem}
\usepackage{graphicx}
\usepackage[margin=1in]{geometry}
\usepackage{hyperref}
\usepackage{tikz}
\usetikzlibrary{cd}
\usepackage{tikz-cd}
\usepackage{xcolor}
\usepackage{ stmaryrd }
\usepackage{mathrsfs}
\usepackage{titling}
\usepackage{mathtools}
\usepackage[english]{babel}
\usepackage{setspace}
\usepackage[normalem]{ulem}
\usepackage{stmaryrd}
\usepackage{bm}

\usepackage{arydshln}

\usetikzlibrary{arrows,positioning,decorations.pathmorphing,
	decorations.markings
}
\tikzset{inner sep=0pt,
	root/.style={circle,draw,minimum size=7pt,thick},
	fatroot/.style={circle,draw,minimum size=10pt,thick},
	short root/.style={circle,fill,minimum size=7pt},
	doublearrow/.style={postaction={decorate},
		decoration={markings,mark=at position .7
			with {\arrow{angle 60}}},double distance=3pt,thick}
}

\DeclareUnicodeCharacter{00A0}{ }

\newtheorem{proposition}{Proposition}

\newtheorem{theorem}[proposition]{Theorem}
\newtheorem{lemma}[proposition]{Lemma}
\newtheorem{corollary}[proposition]{Corollary}

\theoremstyle{definition}
\newtheorem{definition}[proposition]{Definition}
\theoremstyle{remark}

\DeclareMathOperator{\GL}{GL}

\DeclareMathOperator{\Gal}{Gal}

\DeclareMathOperator{\Sym}{Sym}

\DeclareMathOperator{\Pic}{Pic}

\renewcommand{\P}{\mathbb{P}}

\renewcommand{\O}{\mathcal{O}}

\newcommand{\HH}{\mathrm{H}}

\newcommand{\defeq}{\vcentcolon=}
\newcommand{\nc}[2]{\newcommand{#1}{#2}}
\nc{\mc}{\mathcal}
\newcommand{\renc}[2]{\renewcommand{#1}{#2}}
\renc{\C}{\mathbb{C}}
\nc{\ol}{\overline}
\nc{\on}{\operatorname}

\newcommand{\Q}{\mathbb{Q}}
\newcommand{\Z}{\mathbb{Z}}
\newcommand{\F}{\mathbb{F}}

\DeclareMathOperator{\Sel}{Sel}

\DeclareSymbolFont{cyrletters}{OT2}{wncyr}{m}{n}
\DeclareMathSymbol{\Sha}{\mathalpha}{cyrletters}{"58}

\newcommand{\extp}{\@ifnextchar^\@extp{\@extp^{\,}}}
\def\@extp^#1{\mathop{\bigwedge\nolimits^{\!#1}}}
\newcommand{\overbar}[1]{\mkern 1.5mu\overline{\mkern-1.5mu#1\mkern-1.5mu}\mkern 1.5mu}
\DeclareUnicodeCharacter{00A0}{ }
\makeatletter
\def\@tocline#1#2#3#4#5#6#7{\relax
  \ifnum #1>\c@tocdepth 
  \else
    \par \addpenalty\@secpenalty\addvspace{#2}%
    \begingroup \hyphenpenalty\@M
    \@ifempty{#4}{%
      \@tempdima\csname r@tocindent\number#1\endcsname\relax
    }{%
      \@tempdima#4\relax
    }%
    \parindent\z@ \leftskip#3\relax \advance\leftskip\@tempdima\relax
    \rightskip\@pnumwidth plus4em \parfillskip-\@pnumwidth
    #5\leavevmode\hskip-\@tempdima
      \ifcase #1
       \or\or \hskip 1em \or \hskip 2em \else \hskip 3em \fi%
      #6\nobreak\relax
    \dotfill\hbox to\@pnumwidth{\@tocpagenum{#7}}\par
    \nobreak
    \endgroup
  \fi}
\makeatother

\usepackage{tocloft}

\tikzset{
  symbol/.style={
    draw=none,
    every to/.append style={
      edge node={node [sloped, allow upside down, auto=false]{$#1$}}}
  }
}

\newcommand{\height}{\mathrm{ht}}

\usepackage{microtype}

\title{\vspace*{-0.75in}A positive proportion of monic odd-degree hyperelliptic curves of genus $g \geq 4$ have no unexpected quadratic points}

\author{Jef Laga and Ashvin A. Swaminathan}

\begin{document}


\maketitle

\vspace*{-0.2in}
\begin{abstract}
    Let $\mathcal{F}_g$ be the family of monic odd-degree hyperelliptic curves of genus $g$ over $\Q$.
    Poonen and Stoll have shown that for every $g \geq 3$, a positive proportion of curves in $\mathcal{F}_g$ have no rational points except the point at infinity. 
    In this note, we prove the analogue for quadratic points: for each $g\geq 4$, a positive proportion of curves in $\mathcal{F}_g$ have no points defined over quadratic extensions except those that arise by pulling back rational points from $\mathbb{P}^1$.                      

    \end{abstract}


\section{Introduction}

Let $C/\Q$ be a smooth projective curve of genus $g\geq 2$.
Faltings has shown that the set of rational points $C(\Q)$ is finite~\cite[Satz~7]{MR718935}. It is therefore natural to ask how the quantity $\#C(\Q)$ is distributed as $C$ varies in a family of curves. Many authors (see, e.g.,~\cite[Conjecture~2.2]{MR2029869},~\cite{MR2293593},~\cite[Conjecture~1.3(ii)]{MR2340106},~\cite[Conjecture~1]{stoll2009average}) have formulated heuristics or conjectures suggesting that in a family of curves $C$ of genus $g \geq 2$, the quantity $\#C(\Q)$ is very often as small as possible: i.e., most members ought to possess no rational points beyond those that exist generically among all members in the family. Put rather simply, most curves in a family should be as ``pointless'' as can be.

This pithy principle has been proven (or partially proven) for several families of curves of interest; see, e.g.,~\cite{thesource},~\cite{MR3600041} (cf.~\cite[Remark~3.3]{BSSpreprint}),~\cite{MR4258170},~\cite{MR3719247},~\cite{MR4471040}.
Of particular relevance to this note is the major result of Poonen and Stoll \cite{PoonenStoll-Mosthyperellipticnorational}, which concerns the family of monic odd-degree hyperelliptic curves
\begin{align*}
\mathcal{F}_g \defeq \{C\colon y^2 = f(x) = x^{2g+1} + a_1x^{2g} + \dots +a_{2g+1} \mid a_i \in \Z,\, \text{disc}(f)\neq 0\},
\end{align*}
ordered by the height $\height(C) = \max |a_i|^{1/i}$. Here $C$ denotes the unique smooth projective curve with affine equation $y^2 = f(x)$; it has a unique point at infinity $P_{\infty}$, which is $\Q$-rational.
If $g\geq 3$, Poonen and Stoll showed that a positive proportion of curves $C \in \mathcal{F}_g$ satisfy $C(\Q) = \{P_{\infty}\}$, and that this proportion tends to $1$ exponentially fast as $g\rightarrow \infty$.

In this note, we consider quadratic points on curves in the family $\mc{F}_g$. 
Recall that a point $P\in C(\ol{\Q})$ is quadratic if its field of definition $\Q(P)$ is a quadratic extension of $\Q$.
When $C$ is hyperelliptic, it visibly has infinitely many quadratic points: since $C$ has a degree-$2$ map $\pi \colon C \to \mathbb{P}^1$ defined over $\Q$, for all but finitely many points $P \in \mathbb{P}^1(\Q)$ the preimage $\pi^{-1}(P)$ consists of a Galois-conjugate pair of quadratic points.
Explicitly, if $P = (x:1)\in \P^1(\Q)$ then $\pi^{-1}(P)= (x,\pm \sqrt{f(x)})\in C(\ol{\Q})$.
It follows from a more general result of Faltings \cite[Theorem 1 \& Example 4.5]{Faltings-diophantineapproximationabelianvars} that if $C/\Q$ is a hyperelliptic curve of genus $g \geq 4$, then $C$ has only finitely many quadratic points not arising as pullbacks from \mbox{rational points on $\P^1$.}

\subsection{Results}\label{subsec: results}
If $C\in \mathcal{F}_g$, we say a rational point $P$ on $C$ is \emph{expected} if $P = P_{\infty}$, and unexpected otherwise. 
We say a quadratic point $P$ on $C$ is \emph{expected} if it maps to a rational point on $\P^1(\Q)$, and unexpected otherwise.
Our main result shows that when $g\geq 4$, a positive proportion of curves in $\mathcal{F}_g$ have no unexpected rational or quadratic points at all.
Let $\mathcal{F}_{g,X} \defeq \{C \in \mc{F}_g \mid \on{ht}(C) < X\}$, and define the lower density
\begin{align*}
    \delta_{g} \defeq \liminf_{X\rightarrow \infty} \frac{\#\{C\in \mathcal{F}_{g,X} \mid C \text{ has no unexpected rational points or quadratic points}  \}  }{\#\mathcal{F}_{g,X}}.
\end{align*} 
\begin{theorem} \label{thm-main}
    Let $g \geq 4$. When ordered by height, a positive proportion of monic odd hyperelliptic curves $C/\Q$ of genus $g$ possess no unexpected rational points or quadratic points.
    In other words, $\delta_{g}>0$.
\end{theorem}
Theorem \ref{thm-main} is the first unconditional result of its kind to effectivize Faltings' theorem on the finiteness of unexpected quadratic points in a ``large'' family of curves. 

\subsection{Methods}
The proof of Theorem \ref{thm-main} relies on partially generalizing the method of Poonen and Stoll \cite{PoonenStoll-Mosthyperellipticnorational} (later dubbed ``Selmer group Chabauty'' in \cite{Stoll-ChabautywithoutMordellWeil}) to symmetric squares of curves.
Roughly speaking, we proceed in three steps as follows:
\begin{enumerate}
    \item Given a subvariety $X$ of a $g$-dimensional abelian variety $J/\Q$ and a prime number $p$, we prove a criterion guaranteeing that $X(\Q_p) \cap J(\Q)$ is ``small'', under suitable assumptions concerning the disjointness of two sets: (1) the image of $\Sel_p J$ in $J(\Q_p)/pJ(\Q_p)$, and (2) the image of $X(\Q_p)$ under the re-scale and reduce mod-$p$ logarithm map $\rho\log \colon X(\Q_p) \dashrightarrow \P^{g-1}(\F_p)$; see \S\ref{sec-selcha} --- specifically, Proposition \ref{proposition: Chabauty arbitrary subvarieties}. 
    \item Given a curve $C\in \mathcal{F}_g$, we wish to apply the criterion to $J = $ the Jacobian of $C$, $X=$ the image $\Sym^2 C\rightarrow J$, and $p=2$. This leads to a criterion for $C$ to have no unexpected rational or quadratic points; see \S\ref{sec-crit} --- specifically, Corollary \ref{corollary: chabauty criterion for no unexpected points}.
    \item We show that the latter criterion is satisfied for a positive proportion of $C\in \mathcal{F}_g$ when $g\geq 4$ using two inputs: (1) an equidistribution theorem concerning the image of $\Sel_2 J \rightarrow J(\Q_2)/2J(\Q_2)$ as $C$ varies in $\mathcal{F}_g$ due to Bhargava--Gross \cite[Theorem~12.4]{bhargavagross-2selmer}, and (2) the construction of an explicit subfamily $\mc{F}_g^{(2)} \subset \mc{F}_g$ having positive density on which we can control the image $\rho\log (X(\Q_2))$; see \S\ref{sec-exp} --- specifically, Proposition~\ref{prop-boundsym2}.
\end{enumerate}
By calculating the density of $\mc{F}_g^{(2)}$, we obtain explicit lower bounds on the density $\delta_{g}$, proving that $\delta_{g} \geq 2^{-4g^2-6g-7}$; since the curves in $\mathcal{F}_g^{(2)}$ all have good reduction at $2$, the bound must be exponentially small in $g$: indeed, it is exceedingly rare for a hyperelliptic curve in short Weierstrass form to have good reduction at $2$.

If one would like to go beyond Theorem \ref{thm-main} and prove --- analogously to Poonen and Stoll --- that $\delta_g \rightarrow 1$ as $g\rightarrow \infty$, one would need to control the image of $\rho\log\colon X(\Q_2) \dashrightarrow \P^{g-1}(\F_2)$ for a density $1$ set of curves $C\in \mathcal{F}_g$, with arbitrary reduction type at $2$.
The methods of \cite[\S3.2, \S5]{PoonenStoll-Mosthyperellipticnorational} do not directly apply, due to the lack of the Weierstrass preparation theorem for power series over $\Q_p$ in more than one variable. 
It therefore seems that new ideas are needed here.

\subsection{Previous work}

  Using just their result that the average size of $\on{Sel}_2 J$ is at most $3$ in conjunction with an effective version of Chabauty's method, Bhargava and Gross show that a positive proportion of curves in $\mc{F}_g$ have at most $3$ rational points. It is natural to ask whether a similar strategy would work to control points of degree $d$. A version of Chabauty's method for symmetric powers of curves was developed by Siksek~\cite{Sik09}, and Park~\cite{Par16} used to tropical intersection theory to develop an effective version of Siksek's method. However, it was later discovered by Gunther and Morrow~\cite[Remark 5.6]{gunther2019irrational} that Park's work is missing a transversality hypothesis; while it may be feasible to verify this hypothesis for individual curves, it is currently unknown how to do so in a family of curves. Conditional on this transversality hypothesis, Gunther and Morrow prove that a positive proportion of curves in $\mc{F}_g$ have at most 24 unexpected quadratic points for each $g \geq 3$. Our main result, Theorem~\ref{thm-main}, unconditionally reduces their bound for quadratic points to zero when $g\geq 4$.

In a different direction, there are several papers in the literature where versions of Chabauty's method are used to obtain uniform bounds on the number of rational points on curves; here the bounds tend to depend on the genus as well as on the Mordell--Weil rank of the Jacobian. See, e.g., the works of Stoll~\cite{MR3908770} and Katz, Rabinoff, and Zureick-Brown~\cite{MR3566201}. Analogous uniform bounds were obtained for quadratic points by Vemulapalli and Wang~\cite{MR4309843}, but their results depend on the aforementioned work of Park~\cite{Par16} and are thus conditional. 

We end by mentioning the work of Caro and Pasten \cite{CaroPasten-chabautysurfaces}, which has consequences for quadratic points on non-hyperelliptic curves whose Jacobian has Mordell--Weil rank at most one.

\subsection{Acknowledgements}

JL thanks Dick Gross for interesting conversations that motivated us to work on this problem. We also thank
Manjul Bhargava, Jackson Morrow, Bjorn Poonen, Joe Rabinoff, Michael Stoll, Melanie Matchett Wood, and David Zureick-Brown
for helpful conversations, as well as the anonymous referee for their careful reading of the paper.
AS was supported by the National Science Foundation, under the Graduate Research Fellowship as well as Award No.~2202839.

\section{A generalization of \texorpdfstring{$p$}{p}-Selmer-group Chabauty} \label{sec-selcha}

In this section, we generalize the method of $p$-Selmer group Chabauty --- originally developed by Poonen and Stoll to control rational points on a curve embedded inside its Jacobian --- to arbitrary subvarieties of abelian varieties. Specifically, we prove Proposition~\ref{proposition: Chabauty arbitrary subvarieties}, which gives a criterion for a subvariety of an abelian variety to contain only very few rational points.

Let $J/\Q$ be a $g$-dimensional abelian variety, $X\subset J$ a closed subvariety, and $p$ a prime number. 
Let $\overbar{J(\Q)}\subset J(\Q_p)$ be the $p$-adic closure of $J(\Q)$ in $J(\Q_p)$.
Consider the following diagram (modeled on \cite[(6.1)]{PoonenStoll-Mosthyperellipticnorational}):
\begin{equation} \label{eq-tikz} \begin{tikzcd}
	{X(\Q)} && {X(\Q_p)} \\
	{J(\Q)} & {\overbar{J(\Q)}} & {J(\Q_p)} & {\Z_p^g} \\
	{J(\Q)/pJ(\Q)} & {\overbar{J(\Q)}/p\overbar{J(\Q)}} & {J(\Q_p)/pJ(\Q_p)} & {\mathbb{F}_p^g} & {\mathbb{P}^{g-1}(\mathbb{F}_p)} \\
	& {\Sel_p J}
	\arrow[hook, from=1-1, to=1-3]
	\arrow[hook, from=1-1, to=2-1]
	\arrow[hook, from=1-3, to=2-3]
	\arrow[hook, from=2-1, to=2-2]
	\arrow[two heads, from=2-1, to=3-1]
	\arrow[hook, from=2-2, to=2-3]
	\arrow[two heads, from=2-2, to=3-2]
	\arrow["\log", two heads, from=2-3, to=2-4]
	\arrow[two heads, from=2-3, to=3-3]
	\arrow[two heads, from=2-4, to=3-4]
	\arrow["\rho", dashed, from=2-4, to=3-5]
	\arrow["\mu", two heads, from=3-1, to=3-2]
	\arrow["\delta"', hook, from=3-1, to=4-2]
	\arrow[from=3-2, to=3-3]
	\arrow["{\log\otimes \mathbb{F}_p}", from=3-3, to=3-4]
	\arrow["{\mathbb{P}}", dashed, from=3-4, to=3-5]
	\arrow["\eta", from=4-2, to=3-3]
	\arrow["\sigma"', from=4-2, to=3-4]
    \arrow["{\rho \log}", bend left=60, dashed, from=2-3, to=3-5]
	\arrow["{\mathbb{P}\sigma}"', bend right=10, dashed, from=4-2, to=3-5]
\end{tikzcd}\end{equation}
In the diagram~\eqref{eq-tikz}, the unlabeled maps in~\eqref{eq-tikz} are the natural ones, and the dashed arrows refer to partially defined maps. A partially defined map $f\colon Y\dashrightarrow Z$ is a map defined on a subset $S\subset Y$. The image of such a map is defined to be the image of the subset $S$ and is denoted by $f(Y)$; the composition of two partially defined maps is again a partially defined map, defined on the largest subset where evaluation makes sense. The labeled maps $\log$, $\mathbb{P}$, $\rho$, $\delta$, $\eta$, and $\sigma$ are defined as follows:
\begin{itemize}
    \item Write $T_0J = \HH^0(J, \Omega_J^1)^{\vee}$ for the tangent space at the identity of $J$, and write $\log\colon J(\Q_p) \rightarrow T_0J$ for the logarithm map; see \cite[III, \S7.6]{Bourbaki-liegroupsliealgebras} and \cite{Zarhin-logarithmmap} for its definition.
    This is the unique $p$-adic Lie group morphism whose derivative equals the identity. It is a local diffeomorphism, and its kernel equals the (finite) torsion subgroup $J(\Q_p)_{\mathrm{tors}}\subset J(\Q_p)$. The image of $\log$ is a $\Z_p$-lattice in $T_0J$; we fix a basis of $\HH^0(J, \Omega_J^1)$ so that $\log\colon J(\Q_p)\rightarrow T_0J\simeq \Q_p^g$ has image exactly $\Z_p^g$.
    \item For a field $k$ we let $\P\colon k^g\smallsetminus \{0\} \rightarrow \P^{g-1}(k)$ be the usual homogenization map.
    \item We write $\rho$ for the reduction map $\P^{g-1}(\Q_p) \rightarrow \P^{g-1}(\F_p)$ or the composite $\Q_p^g \smallsetminus \{0\} \xrightarrow{\P} \P^{g-1}(\Q_p) \rightarrow \P^{g-1}(\F_p)$.
    \item The $\F_p$-vector space $\Sel_pJ$ is the $p$-Selmer group of $J$, and it comes equipped with a Kummer map $\delta\colon J(\Q)/pJ(\Q) \rightarrow \Sel_p J$.
If $x\in \Sel_p J \subset \HH^1(\Q,J[p])$, then the local Galois cohomology class $x_p$ at $p$ lies in the image of the injective map $\delta_p \colon J(\Q_p)/pJ(\Q_p) \to \HH^1(\Q_p, J[p])$.
The map $\eta$ is given by $x\mapsto \delta_p^{-1}(x_p)$.
The map $\sigma$ is the composition of $\eta$ with $\log\otimes \F_p$.
\end{itemize}
All the squares and triangles in this diagram are commutative, whenever compositions are defined.

Observe that $X(\Q)$ lies in the intersection $X(\Q_p) \cap \ol{J(\Q)} \subset J(\Q_p)$. In  Chabauty's method, the objective is to bound $X(\Q)$ by controlling this intersection. The idea of $p$-Selmer group Chabauty is to achieve this control by pushing this intersection forward all the way to $\mathbb{P}^{g-1}(\F_p)$ via the maps in the diagram~\eqref{eq-tikz}. Specifically, we obtain the following result:

\begin{proposition}\label{proposition: Chabauty arbitrary subvarieties}
Suppose $\sigma$ is injective and $\P\sigma(\Sel_p J) \cap \rho\log(X(\Q_p))  = \varnothing$.
Then $X(\Q_p) \cap \overbar{J(\Q)}\subset J(\Q_p)[p']$, where $J(\Q_p)[p']$ denotes the $($finite$)$ subgroup of prime-to-$p$ torsion elements of $J(\Q_p)$.
\end{proposition}
\begin{proof}
    The proof is essentially identical to that of \cite[Proposition 6.2]{PoonenStoll-Mosthyperellipticnorational}, which never uses the assumption that $X$ is a curve; we include the details for completeness.
    By assumption, $\sigma$ is injective, so $\sigma\delta$ is also injective.
    It then follows from~\eqref{eq-tikz} that $\mu$ is injective too. Thus, $\mu$ is an isomorphism, and the map $\overbar{J(\Q)}/p\overbar{J(\Q)}\rightarrow \F_p^g$ is injective.

    Suppose that the conclusion fails, and let $P\in X(\Q_p) \cap \overbar{J(\Q)}$ be an element not contained in $J(\Q_p)[p']$.
    Then $P$ is not infinitely $p$-divisible in $J(\Q_p)$, and hence also not  infinitely $p$-divisible in $\overbar{J(\Q)}$. Thus, there is an element $Q \in \overbar{J(\Q)}$ and an integer $n\geq 0$ such that $P = p^n Q$ and $Q\not\in p \overbar{J(\Q)}$.
    The image $\ol{Q}$ of $Q$ in $\overbar{J(\Q)}/p \overbar{J(\Q)}$ is nonzero, so its image under $\log\otimes \F_p$ in $\F_p^g$ is nonzero.
    Consequently, $\P\sigma \delta\mu^{-1}(\ol{Q})$ is well-defined and lies in $\P \sigma (\Sel_p J)$.
    On the other hand, 
    \begin{align*}
        \P\sigma \delta\mu^{-1}(\ol{Q}) = \rho \log(Q) = \rho \log(p^n Q) = \rho \log (P) \in \rho \log(X(\Q_p)),
    \end{align*}
    contradicting our assumption that $\P \sigma (\Sel_p J) \cap \rho \log(X(\Q_p))  = \varnothing$.
\end{proof}


\section{A criterion for every rational and quadratic point to be expected} \label{sec-crit}

In this section, we specialize the content of \S\ref{sec-selcha} to the case of Jacobians of monic odd-degree hyperelliptic curves over $\Q$, and we obtain a criterion --- namely, Corollary~\ref{corollary: chabauty criterion for no unexpected points} --- for every rational point and every quadratic point on such a curve to be expected.

We start slightly more generally and introduce some notation that we will use in the remainder of the paper.
Let $C$ be a hyperelliptic curve of genus $g\geq 2$.
Denote its Jacobian variety by $J$. There is by definition a degree-$2$ map $\pi \colon C \to \mathbb{P}^1$ defined over $\Q$, called the hyperelliptic map, which is unique up to automorphism of $\mathbb{P}^1$. 
We call the pullback along $\pi$ of a $\Q$-point of $\P^1$ a hyperelliptic divisor, and the linear equivalence class of such a divisor the hyperelliptic divisor class $H$.
The associated linear system is denoted $|H| \defeq \P(\HH^0(C,\mathcal{O}_C(H))) \simeq \P^1$, and consists of divisors of the form $\pi^*(t)$, where $t$ is a point of $\P^1$.
(When $C \in \mathcal{F}_g$, a convenient choice for $H$ is the class $[2P_{\infty}]$.)

The symmetric square $\on{Sym}^2 C$ is a smooth, projective, and geometrically integral variety parametrizing effective divisors of degree $2$ on $C$.
We may therefore view $|H|$ as a closed subscheme of $\on{Sym}^2 C$. 
Let $\on{AJ}\colon \Sym^2 C\rightarrow J$ be the Abel--Jacobi morphism sending an effective degree-$2$ divisor $D$ to the divisor class $[D]-H$.
Let 
\begin{align}\label{equation: definition X}
    X  \defeq \on{AJ}(\on{Sym}^2 C)
\end{align}
be the scheme-theoretic image of $\on{Sym}^2C$ under $\on{AJ}$. The tangent space of the point $0$ in $X$ has dimension $g$ \cite[Chapter IV, Proposition (4.2)]{arbarellocornalbagriffithsharris-geometryalgebraiccurves}, so if $g\geq 3$, then $X$ has a singular point, resolved by the map $\on{AJ} \colon \on{Sym}^2 C \rightarrow X$.
\begin{lemma}\label{lemma: AJ birational onto image}
    The morphism $\mathrm{AJ}\colon \on{Sym}^2 C \rightarrow J$ is birational onto its image $X$.
    It is an isomorphism above $X\smallsetminus \{0\}$ and the $($scheme-theoretic$)$ fiber above $\{0\}$ is the closed subscheme $|H| \subset \on{Sym}^2 C$.
\end{lemma}
\begin{proof}
    The fiber above $[E] \in J$ is the projective space of divisors $D\in \on{Sym}^2 C$ in the linear system $|E+H|$ by \cite[Chapter IV, Lemma (1.1)]{arbarellocornalbagriffithsharris-geometryalgebraiccurves}.
    Since $|H|$ is the only positive dimensional degree-$2$ linear system, the lemma follows.
\end{proof}
\begin{lemma}\label{lemma: X(Q) empty means no unexpected points}
    Suppose that $X(\Q) = \{0\}$.
    Then $C(\Q)$ consists of at most one Weierstrass point, and every quadratic point on $C$ maps to $\P^1(\Q)$ under $\pi$.
\end{lemma}
\begin{proof}
    Using Lemma \ref{lemma: AJ birational onto image}, the condition $X(\Q) = \{0\}$ implies that $(\on{Sym}^2C)(\Q) = |H|(\Q)$. 
    In other words, every $\Q$-rational effective degree-$2$ divisor on $C$ is of the form $\pi^*(t)$ for some $t\in \P^1(\Q)$.
    One can check that this implies the conclusions of the lemma.
\end{proof}

Applying Proposition~\ref{proposition: Chabauty arbitrary subvarieties} to our choice of $X$ and using Lemma~\ref{lemma: X(Q) empty means no unexpected points}, we obtain the following criterion for the absence of unexpected rational and quadratic points on curves $C\in \mathcal{F}_g$ in terms of the diagram~\eqref{eq-tikz}:
\begin{corollary}\label{corollary: chabauty criterion for no unexpected points}
    Let $C\in \mathcal{F}_g$. Suppose that: 
    \begin{itemize}
        \item $\sigma\colon \Sel_p J \rightarrow \F_p^g$ is injective;
        \item $\P\sigma(\Sel_p J) \cap \rho \log (X(\Q_p)) = \varnothing$; and
        \item $X(\Q_p) \cap J(\Q_p)[p'] =\{0\}$.
    \end{itemize}
    Then $C$ has no unexpected rational or quadratic points.
\end{corollary}

\section{Controlling the image \texorpdfstring{$(\Sym^2 C)(\Q_2)$}{Sym2} under \texorpdfstring{$\rho\log$}{rho log}}

The purpose of this section is to construct a subfamily $\mathcal{F}_g^{(2)}\subset \mathcal{F}_g$ of positive density for which (in the notation of \S\ref{sec-crit}) $X(\Q_2) \cap J(\Q_2)[2'] = \{0\}$ and for which we can control the image $\rho\log(X (\Q_2))$.

In~\cite{PoonenStoll-Mosthyperellipticnorational}, which concerns rational points on $C$, the corresponding problem is to control the set $\rho\log(C(\Q_2))$. To do this, Poonen and Stoll consider a particular curve $C_0/\Q_2$ --- namely, the curve with long Weierstrass equation $C_0\colon y^2 + y = x^{2g+1} + x + 1$ --- and they prove for this choice of $C_0$ that $\#\rho\log(C_0(\Q_2)) = 1$. 
They then separately prove that as $C$ varies in the moduli space of monic odd hyperelliptic curves of genus $g$ over $\Q_2$, the set $\rho\log(C(\Q_2))$ is locally constant, from which it follows that for every $g\geq 3$ a positive proportion of curves $C \in \mc{F}_g$ have $\#\rho\log(C(\Q_2)) = 1$.

Here, we adopt a slightly different approach: we construct a subfamily $\mc{F}_g^{(2)} \subset \mc{F}_g$, defined by congruence conditions, such that the image $\rho\log((\on{Sym}^2 C)(\Q_2))$ can be explicitly bounded for \emph{all} curves $C \in \mc{F}_g^{(2)}$, much as Poonen and Stoll did for the single curve $C_0$.  All curves in these subfamilies have good reduction at $2$, and their Jacobians all have trivial $2$-torsion over both $\Q_2$ and $\F_2$. Thus, while $\mc{F}_g^{(2)}$ has positive density in $\mc{F}_g$, this density is exponentially small in $g$.
\medskip

The curves in $\mc{F}_g^{(2)}$ will be those that arise by completing the square in long Weierstrass equations considered in the following definition:
\begin{definition}
    We say a monic polynomial $h \in \Z_2[x]$ of degree $2g+1$ is \emph{good} if it satisfies the following property: letting $S$ be the set of exponents of the mod-$2$ reduction of $h$, and for each $i \in \{0,1,2\}$, letting $S_i \subset S$ be the subset of exponents congruent to $i \pmod 3$, we have $S \supset \{0,2g+1\}$, $\#S_0 \equiv 0 \pmod 2$, and $\#(S_1 \sqcup S_2) \equiv 1 \pmod 2$.
    For such an $h$, let $C_h/\Q_2$ be the hyperelliptic curve with long Weierstrass model 
    \begin{equation} \label{eq-egcurveq}
    y^2+y = h(x).
    \end{equation}
\end{definition}

Let $P_g = \{ h(x) = x^{2g+1} + c_1 x^{2g} + \cdots +c_{2g+1} \in \Z_2[x]\} = \Z_2^{2g+1}$ be the parameter space of monic polynomials of degree $2g+1$ over $\Z_2$, and let $U_g\subset P_g$ be the open subset of good polynomials.

Let $C = C_h$, where $h\in U_g$ is good. 
The model \eqref{eq-egcurveq} has good reduction, hence defines a smooth proper model $\mathcal{C}$ of $C$ over $\Z_2$.
Let $\mathcal{J} = \Pic^0_{\mathcal{C}/\Z_2}$ be the N\'eron model of the Jacobian $J$ of $C$.
We write $P_{\infty} \in \mathcal{C}(\Z_2)$ for the unique point at infinity. 

The next lemma records the useful properties satisfied by good curves:
\begin{lemma} \label{lem-goodness}
    Let $C/\Q_2$ be a hyperelliptic curve of the form $C_h$ for some good $h\in \Z_2[x]$.
    Then:
    \begin{enumerate}
        \item[$(a)$] $J(\Q_2)[2] = \{0\} = \mc{J}(\F_2)[2]$.
        \item[$(b)$] $\mc{C}(\F_2) = \{P_\infty\}$ and $\mc{C}(\F_4) = \{P_\infty,(0,\alpha),(0,\alpha+1),(1,\alpha),(1,\alpha+1)\}$, where $\alpha \in \F_4 \smallsetminus \{1\}$ denotes a nontrivial third root of unity.
        \item[$(c)$] $X(\Q_2) \cap J(\Q_2)[2'] = \{0\}$, where $X$ is defined by \eqref{equation: definition X} and $J(\Q_2)[2']$ is the prime-to-$2$ torsion subgroup of $J(\Q_2)$.
    \end{enumerate}
\end{lemma}
\begin{proof}
     Completing the square on the left-hand side of~\eqref{eq-egcurveq}, we obtain a short Weierstrass equation for $C$ of the form $y^2 = f(x) \defeq h(x) + 1/4$, thus realizing $C$ has a monic odd-degree hyperelliptic curve.
     
     To verify $(a)$, note that the Newton polygon of $f$ is the line segment from $(0,-2)$ to $(2g+1,0)$, which contains no lattice points, implying that $f$ is irreducible over $\Q_2$. Consequently, we have $J(\Q_2)[2] = 0$. Since the hyperelliptic involution on $\mc{C}_{\F_2}$ is given by $(x,y) \mapsto (x,y+1)$, its only fixed point is $P_\infty \in \mc{C}(\F_2)$, implying that $\mc{J}(\ol{\F}_2)[2] = 0$, hence $\mc{J}(\F_2)[2]=0$. 
     The verification of $(b)$ is a computation.
     To verify $(c)$, note that the restriction of the reduction map $\on{red}\colon J(\Q_2)[2'] \rightarrow \mathcal{J}(\F_2)$ is injective, so it suffices to prove that $\on{red}(X(\Q_2) \cap J(\Q_2)[2']) = \{0\}$.
     Every element in $\on{red}(X(\Q_2) \cap J(\Q_2)[2'])$ is represented by a divisor of the form $P_1 + P_2 - 2P_{\infty}$, where either $P_1, P_2 \in \mathcal{C}(\F_2)$, or $P_1,P_2$ form a pair of Galois conjugate points in $\mathcal{C}(\F_4)$.
     In both cases, the divisor class $[P_1 + P_2 - 2P_{\infty}]$ is trivial by $(b)$.
\end{proof}

Lemma~\ref{lem-goodness} allows us to choose a basis of differentials uniformly for all $C_h$ such that the logarithm map is surjective. 
If $h\in U_g$ and $C= C_h$ is the hyperelliptic curve with equation \eqref{eq-egcurveq}, the change-of-variables $s \defeq 1/x$ and $t \defeq y/x^{g+1}$ applied to~\eqref{eq-egcurveq} yields the equation
\begin{equation} \label{eq-stcurveq}
t^2 + s^{g+1}t = s^{2g+2}h(1/s).
\end{equation}
Define
\begin{align} \label{eq-omega1def}
\omega_1 & \defeq \frac{dt}{\frac{\partial}{\partial s}\left(t^2 + s^{g+1}t - s^{2g+2}h(1/s)\right)}, 
\end{align}
and $\omega_j \defeq s^{j-1}\omega_1$ for each $j \in \{2, \dots, g\}$.
By the general theory of hyperelliptic curves, $\omega_1,\dots,\omega_g$ is a $\Z_2$-basis of $\HH^0(\mathcal{C}, \Omega^1_{\mathcal{C}/\Z_2})$. 
By Lemma~\ref{lem-goodness} in conjunction with~\cite[Lemma~10.1]{PoonenStoll-Mosthyperellipticnorational}, the logarithm map $\log \colon J(\Q_2) \rightarrow \HH^0(J,\Omega_J^1)^{\vee} \simeq  \HH^0(C, \Omega^1_C)^{\vee} \simeq \Q_2^g$ corresponding to this choice of basis satisfies $\log(J(\Q_2)) = (2\Z_2)^g$.
Using this identification, we may consider the subset $\rho \log(X(\Q_2)) \subset \P^{g-1}(\F_2)$ like we did in \S\ref{sec-selcha} and \S\ref{sec-crit}.

\begin{proposition} \label{prop-boundsym2}
    As $h$ varies in $U_g$, the subset $\rho \log(X(\Q_2))\subset \P^{g-1}(\F_2)$ is locally constant and satisfies 
    \begin{align*}
    \#\rho \log(X(\Q_2))\leq 5.    
    \end{align*}
\end{proposition}
\begin{proof}
Let $h\in U_g$ be good and $C = C_h$.
Since $\rho\log(X(\Q_2)) = \rho\log(\mathrm{AJ}((\Sym^2C)(\Q_2)))$, it suffices to determine the latter set.
Given a finite extension $K/\Q_2$ with residue field $k$ and an element $x\in \mathcal{C}(k)$, let $D_x(K)$ be the fiber of the reduction map $C(K) \rightarrow \mathcal{C}(k)$. 
Similarly if $x \in (\Sym^2\mathcal{C})(k)$, let $D_{x}(K)$ be the fiber of the reduction map $(\Sym^2 C)(K) \rightarrow (\Sym^2\mathcal{C})(k)$ above $x$.
Lemma \ref{lem-goodness}$(b)$ shows that $(\Sym^2\mathcal{C})(\F_2)$ consists of $3$ elements, represented by the degree-$2$ divisors $2P_{\infty}, (0, \alpha) + (0, \alpha+1)$ and $(1, \alpha)+ (1, \alpha+1)$.
We bound $\rho \log(X(\Q_2))$ by computing $\rho\log$ explicitly on each residue polydisk $D_x(\Q_2)$ for $x\in (\Sym^2\mathcal{C})(\F_2)$.

\medskip
\noindent \emph{Residue polydisk centered at $2P_\infty$.} 
Take a point $P + Q \in D_{2P_\infty}(\Q_2)$. There are two possibilities: either $P,\,Q$ are both defined over $\Q_2$, or $P,\,Q$ are conjugates over a quadratic extension of $\Q_2$. Let $K$ denote the field of definition of $P$, and let $\pi$ be a uniformizer for $K$. 
We use the model \eqref{eq-stcurveq} of $\mathcal{C}$.
In that model, the point $P_\infty \in \mc{C}(\F_2)$ has coordinates $(s,t) = (0,0)$, and the hyperelliptic involution $\iota\colon \mc{C}\rightarrow\mc{C}$ has the form $\iota(s,t) = (s, s^{g+1}-t)$.
The coordinate $t \defeq y/x^{g+1}$ is a uniformizer at $P_{\infty}$ and remains so when reduced mod $\pi$, so the assignment $R\mapsto t(R)$ induces a bijection $D_{P_\infty}(K) \simeq \pi \mc{O}_K$.

We now compute the restriction of the logarithm map to $D_{P_\infty}(K)$ explicitly as a $g$-tuple of power series in $t$. 
Using~\eqref{eq-stcurveq}, we can expand $s$ as a power series in $t$ centered at $P_\infty$; this expansion has $\Z_2$-coefficients and is of the form $s = t^2 + 2ct^3 + \cdots$ for some $c \in \Z_2$. Substituting this expansion of $s$ into~\eqref{eq-omega1def}, we find that $\omega_1 = (1 + 0t + 0t^2 +2 c't^3 + \cdots)dt$ for some $c' \in \Z_2$, and hence that $\omega_j = (t^{2j-2} + 2c(j-1)t^{2j-1} + \cdots)dt$ for each $j \in \{1, \dots, g\}$. 
By the theory of Coleman integration \cite[\S II]{coleman-integrals}, if $R\in D_{P_{\infty}}(K)$ has coordinate $t \in \pi \mathcal{O}_K$, then $\log(R-P_{\infty})=\big(\hspace*{-1pt}\int_0^t\omega_1,\dots,\int_0^t\omega_g\big)$.

Denote the coordinates of the points $P, Q, \iota(Q) \in D_{P_{\infty}}(K)$ by $t_1, t_2, t_3 \in \pi \mathcal{O}_K$ respectively. 
Since $[Q+\iota(Q)-2P_{\infty}]=0$ in $J(\Q_2)$, we have
$\log(P+Q-2P_\infty) =\log(P-P_{\infty}) - \log(\iota(Q)-P_{\infty})
$, hence
\begin{equation} \label{eq-sumphi}
\log(P+Q-2P_\infty)  = \left((t_1 - t_3) + \cdots, \frac{t_1^3-t_3^3}{3} + \cdots, \dots, \frac{t_1^{2g-1} - t_3^{2g-1}}{2g-1} + \cdots\right).
\end{equation}
Dividing each component on the right-hand side of~\eqref{eq-sumphi} by $t_1 - t_3$, we deduce that $\rho \log(P+Q-2P_{\infty})$ is given by $\rho(1 + \cdots, \cdots)$; here, the ellipses are used to abbreviate terms of the form 
\begin{equation} \label{eq-sampleterm}
\frac{c''}{i+1} \cdot \big(t_1^i + t_1^{i-1}t_3 + \cdots + t_1 t_3^{i-1} + t_3^{i-1}\big),
\end{equation}
where $c'' \in \Z_2$ for all $i \geq 1$ and $c'' \in 2\Z_2$ when $i \in \{1, 3\}$. For $t_1,\,t_3 \in \pi \mc{O}_K$, the valuation of the expression~\eqref{eq-sampleterm} is always positive. We conclude that the map $\rho \log$ is constant on $D_{2P_\infty}(\Q_2)$, with value $(1 : 0 : \cdots : 0)$.

\medskip
\noindent \emph{Residue polydisk above $(0,\alpha) + (0,\alpha+1)$.} 
In this case, we use the model  \eqref{eq-egcurveq}.
Let $K \defeq \Q_4$ be the unramified quadratic extension of $\Q_2$. Any point in $D_{(0,\alpha) + (0,\alpha+1)}(\Q_2)$ is of the form $P + \ol{P}$, where $P \in C(K)$ reduces to $(0,\alpha)$ and $\ol{P}$ is the $\on{Gal}(K/\Q_2)$-conjugate of $P$. Let $\gamma\in \O_K$ be an element with $\gamma^2 + \gamma = h(0)$; such an element exists by Hensel's lemma and reduces to $\alpha\in \F_4$. 
Let $P_0 \defeq (0, \gamma) \in C(K)$. 
Then the $\on{Gal}(K/\Q_2)$-conjugate $\ol{P}_0 = (0, -1-\gamma) \in C(K)$ is in fact the image of $P_0$ under the hyperelliptic involution $\iota$.
The coordinate $x$ is a uniformizer at both $P_0$ and $\ol{P}_0$, giving identifications $D_{(0, \alpha)}(K) \simeq 2\mc{O}_K$ and $D_{(0, \alpha+1)}(K) \simeq 2\mc{O}_K$. 

By the general theory of hyperelliptic curves, the elements
\begin{equation} \label{eq-diffsinx}
    \eta_1 \defeq \frac{dx}{2y+1} = \frac{dy}{-h'(x)},
\end{equation}
and $\eta_j \defeq x^{j-1}\eta_1$ for $j\in \{2,\dots, g\}$ form a $\Z_2$-basis of $\HH^0(\mc{C}, \Omega^1_{\mc{C}/\Z_2}\big)$. Using~\eqref{eq-egcurveq}, we can expand $y$ as a power series in $x$ on $D_{(0, \alpha)}(K)$ centered at $P_0$; this expansion has $\mc{O}_K$-coefficients and constant term $\gamma$.
  Substituting this expansion of $y$ into~\eqref{eq-diffsinx}, we find a power series expansion of $\eta_1/dx$ with $\mc{O}_K$-coefficients and constant coefficient $(2\gamma+1) \in \O_K^{\times}$. 
  Now if $R\in D_{(0, \alpha)}(K)$ corresponds to $x\in 2\mathcal{O}_K$, then $\phi(x)\defeq \log(P - P_{\infty}) =  \log(P_0-P_{\infty}) + \big(\hspace*{-1pt}\int_0^x\eta_1,\dots,\int_0^x\eta_g\big)$. 

  Let the points $P \in D_{(0, \alpha)}(K)$ and $\ol{P}\in D_{(0, \alpha+1)}(K)$ have coordinates $x_1, x_2 \in 2\mathcal{O}_K$ respectively.  
  Then $x_2$ is the $\Gal(K/\Q_2)$-conjugate of $x_1$, and the hyperelliptic conjugate $\iota(\ol{P})\in D_{(0, \alpha)}(K)$ has coordinate $x_2$.
  Since the hyperelliptic involution acts as $-1$ on the Jacobian, we have $\log(P  +\ol{P} -2P_{\infty}) = \log(P-P_{\infty}) - \log(\iota(\ol{P}) - P_{\infty}) = \phi(x_1) - \phi(x_2)$, hence
\begin{equation} \label{eq-logonprod}
\log(P+ \ol{P} -2P_{\infty}) =  \left(\int_0^{x_1} \eta_1 - \int_0^{x_2} \eta_1, \dots, \int_0^{x_1} \eta_g - \int_0^{x_2} \eta_g\right).
\end{equation}
Each difference of integrals in the $g$-tuple on the right-hand side of~\eqref{eq-logonprod} is a sum of terms of the form
\begin{equation*}
    \frac{c}{i+1} \cdot (x_1^{i+1} - x_2^{i+1})
\end{equation*}
where $c \in \Z_2$.
For each $j \in \{1,\dots, g\}$, the smallest value of $i$ occurring in the $j^{\mathrm{th}}$ component is exactly $j-1$, in which case $c =2\gamma+1$ lies in $\Z_2^{\times}$. Dividing each component on the right-hand side of~\eqref{eq-logonprod} by $(2\gamma+1)\cdot (x_1 - x_2)$, we deduce that $\rho\log(P+\ol{P}-2P_{\infty}) = \rho(1+\cdots, (x_1+x_2)/2+\cdots, \cdots)$, where all the power series in the latter $g$-tuple are sums of terms of the form
\begin{equation} \label{eq-thegeneralform2}
    \frac{c}{i+1} \cdot (x_1^{i} + x_1^{i-1}x_2 + \cdots + x_1x_2^{i-1} + x_2^{i}),
\end{equation}
and where once again, $c \in \Z_2$ and for each $j \in \{1,\dots, g\}$, the smallest value of $i$ occurring in the $j^{\mathrm{th}}$ component is $j-1$. 
For $x_1, \, x_2 \in 2\mc{O}_K$, the valuation of the expression~\eqref{eq-thegeneralform2} is positive so long as $i \geq 2$. 
Therefore $\rho\log(P + \ol{P} -2P_{\infty}) = \rho(1,(x_1+x_2)/2,0,\cdots,0)$.
%
We conclude that, in the basis $\eta_1, \dots, \eta_g$, we have $\rho \log(D_{(0, \alpha)+(0, \alpha+1)}(\Q_2)) = \{(1 : 0 : 0 : \cdots : 0),(1 : 1 : 0 : \cdots : 0)\}$.

\medskip
\noindent \emph{Residue polydisk above $(1,\alpha) + (1,\alpha+1)$.} Under the transformation $x \mapsto x-1$, the conditions on $h$ in the proposition are preserved. Thus, the analysis for this case is identical to that of the previous case.

\medskip
Putting the three cases together, we obtain $\#\rho\log(X(\Q_2)) \leq 1+2+2 = 5$.
To show that $\rho\log(X(\Q_2))\subset \P^{g-1}(\F_2)$ is locally constant as $h$ varies in $U_g$, it suffices to show this on each residue polydisk separately. 
This is obvious for $D_{2P_\infty}(\Q_2)$.
For the residue disk $D_{(0, \alpha) + (0, \alpha+1)}(\Q_2)$, there exists a matrix $A_h \in \GL_g(\Z_2)$ that transforms the basis $\eta_1, \dots, \eta_g$ of $\HH^0(\mathcal{C}, \Omega^1_{\mathcal{C}/\Z_2})$ into $\omega_1, \dots, \omega_g$.
Since the association $h\mapsto A_h$ is continuous and $\rho \log (D_{(0, \alpha) + (0, \alpha+1)}(\Q_2))$ is constant with respect to the basis $\eta_1, \dots, \eta_g$, $\rho \log (D_{(0, \alpha) + (0, \alpha+1)}(\Q_2))$ is locally constant with respect to the basis $\omega_1, \dots, \omega_g$ as $h\in U_g$ varies.
The argument for $D_{(1,\alpha) + (1,\alpha+1)}(\Q_2)$ is analogous.
\end{proof}

Let $\mc{F}_g^{(2)}\subset \mc{F}_g$ be the subset of polynomials $f$ of the form $4^{2g+1}(h(x/4)+1/4)$ for some $h\in U_g \cap \Z[x]$.
Completing the square shows that monic odd hyperelliptic curves with equation $y^2 = f(x)$ with $f\in \mc{F}_g^{(2)}$ are isomorphic to hyperelliptic curves with of the form \eqref{eq-egcurveq} for some good $h\in U_g$.

\begin{lemma} \label{lem-bound2}
       The subfamily $\mc{F}_g^{(2)} \subset \mc{F}_g$ is defined by congruence conditions modulo $2^{4g+2}$ and has density $2^{-4g^2-6g-5}$.
\end{lemma}
\begin{proof}
\nc{\wt}{\widetilde}
    Equip the parameter space $P_g = \Z_2^{2g+1}$ of all monic polynomials in $\Z_2[x]$ of degree $2g+1$ with the probability Haar measure. 
    A combinatorial calculation shows that the open subset of good polynomials $U_g\subset P_g$ has measure $1/8$. 
    On the other hand, the $\Z_2$-linear map $h\mapsto 4^{2g+1}h(x/4)$ on $P_g$ has determinant $4  =4^{2g^2+3g+1}$, so the subset of $f\in P_g$ of curves of the form $4^{2g+1}h(x/4) + 4^{2g}$ for some $h\in P_g$ has measure $4^{-2g^2-3g-1}$.
    We conclude that the measure of the subset $\{4^{2g+1}h(x/4)+4^{2g} \mid h\in U_g\}\subset P_g$ is $(1/8) \times 4^{-2g^2-3g-1} = 2^{-4g^2-6g-5}$, so the density of $\mathcal{F}_g^{(2)}$ is also $2^{-4g^2-6g-5}$.
\end{proof}

\section{Proof of Theorem~\ref{thm-main}} \label{sec-exp}

The proof of Theorem \ref{thm-main} is completed by the next proposition.
Just as in \cite{PoonenStoll-Mosthyperellipticnorational}, a crucial ingredient in the proof is an equidistribution result of Bhargava--Gross \cite[Theorem 12.4]{bhargavagross-2selmer}.
If $S\subset T\subset \mathcal{F}_g$ are subsets, the relative density of $S$ in $T$ is by definition the quantity 
$\lim_{X\rightarrow \infty} \#(S \cap \mathcal{F}_{g,X}) /\#(T \cap \mathcal{F}_{g,X})$,
whenever it exists. 
Define the relative lower and upper density of $S$ in $T$ by replacing the limit by $\liminf$ and $\limsup$, respectively.
\begin{proposition}
    The relative lower density of the set
    \begin{align}\label{equation: subset of good curves with no points}
        \{ C\in \mathcal{F}_g^{(2)} \mid C \text{ has no unexpected rational or quadratic points}\}
    \end{align}
    in $\mc{F}_g^{(2)}$ is at least $1-6\times 2^{1-g}$.
    Consequently, if $g\geq 4$, a positive proportion of curves in $\mathcal{F}_g$ have no unexpected rational or quadratic points.
\end{proposition}
\begin{proof}
    This is a standard adaptation of \cite[Proposition~8.13]{PoonenStoll-Mosthyperellipticnorational} to our setting.
    Proposition \ref{prop-boundsym2} shows that $\mathcal{F}_g^{(2)}$ is partitioned into finitely many subsets $\mathcal{G}_1, \dots, \mathcal{G}_m$, each defined by congruence conditions modulo a power of $2$, such that the subset $\rho\log(X(\Q_2))\subset \P^{g-1}(\F_2)$ is constant when $h$ varies in $\mathcal{G}_i$.
    Fix such an $i\in \{1,\dots,m\}$.
    If $C\in \mathcal{G}_i$, then $J(\Q_2)[2]=0$ by Lemma \ref{lem-goodness}$(a)$, so the logarithm map defined using the basis \eqref{eq-omega1def} induces an isomorphism $J(\Q_2)/2J(\Q_2) \simeq\F_2^g$.
    Write $\sigma \colon \Sel_2J\rightarrow \F_2^g$ for the corresponding composition $\Sel_2 J\rightarrow J(\Q_2)/2J(\Q_2)\rightarrow \F_2^g$ considered in \S\ref{sec-crit}.
    Bhargava--Gross have shown~\cite[Theorem~12.4]{bhargavagross-2selmer} that nontrivial elements of $\Sel_2J$ equidistribute under $\sigma$: when $C$ varies in $\mathcal{G}_i$ and $I\subset \F_2^g$ is a subset, they show that the average size of $\{s \in \Sel_2J\setminus \{0\} \colon  \sigma(s) \in I\}$ is at most $\#I\times 2^{1-g}$. 
    It follows that the relative upper density of curves $C\in \mathcal{G}_i$ such that $\{s \in \Sel_2J \setminus \{0\}\colon \sigma(s) \in I\}$ is nonempty is at most $\#I\times 2^{1-g}$.
    Applying this to $I = \P^{-1}(\rho\log(X(\Q_2))\cup \{0\})$ and using Proposition \ref{prop-boundsym2}, we see that the relative lower density of curves $C\in \mathcal{G}_i$ such that the first two conditions of Corollary \ref{corollary: chabauty criterion for no unexpected points} are satisfied is at least $1-\#I\times 2^{1-g} \geq 1-6\times 2^{1-g}$.
    The third condition of Corollary \ref{corollary: chabauty criterion for no unexpected points} is always satisfied for $C\in \mathcal{F}_g^{(2)}$ by Lemma \ref{lem-goodness}$(c)$.
    We conclude that the intersection of the set \eqref{equation: subset of good curves with no points} with $\mathcal{G}_i$ has relative lower density at least $1-6\times 2^{1-g}$ in $\mathcal{G}_i$.
    Summing over $i$, the set \eqref{equation: subset of good curves with no points} itself has relative lower density at least $1-6\times 2^{1-g}$ in $\mathcal{F}_g^{(2)}$.
    
    Thus, the lower density of curves in $\mathcal{F}_g$ that have no unexpected rational or quadratic points is at least $(1-6\times 2^{1-g})\times (\text{density of } \mathcal{F}_g^{(2)}) = (1-6\times 2^{1-g}) \times 2^{-4g^2-6g-5}$ by Lemma \ref{lem-bound2}.
    This is positive when $g\geq 4$. 
\end{proof}

\bibliographystyle{alpha}

\begin{thebibliography}{ACGH85}

\bibitem[ACGH85]{arbarellocornalbagriffithsharris-geometryalgebraiccurves}
E.~Arbarello, M.~Cornalba, P.~A. Griffiths, and J.~Harris.
\newblock {\em Geometry of algebraic curves. {V}ol. {I}}, volume 267 of {\em
  Grundlehren der mathematischen Wissenschaften [Fundamental Principles of
  Mathematical Sciences]}.
\newblock Springer-Verlag, New York, 1985.

\bibitem[BG13]{bhargavagross-2selmer}
{M}. Bhargava and {B}.~{H}. Gross.
\newblock The average size of the 2-{S}elmer group of {J}acobians of
  hyperelliptic curves having a rational {W}eierstrass point.
\newblock In {\em Automorphic representations and {$L$}-functions}, volume~22
  of {\em Tata Inst. Fundam. Res. Stud. Math.}, pages 23--91. Tata Inst. Fund.
  Res., Mumbai, 2013.

\bibitem[BGW17]{MR3600041}
M.~Bhargava, B.~H. Gross, and X.~Wang.
\newblock A positive proportion of locally soluble hyperelliptic curves over
  {$\Bbb Q$} have no point over any odd degree extension.
\newblock {\em J. Amer. Math. Soc.}, 30(2):451--493, 2017.
\newblock With an appendix by Tim Dokchitser and Vladimir Dokchitser.

\bibitem[Bha13]{thesource}
{M}. Bhargava.
\newblock Most hyperelliptic curves over $\mathbb{Q}$ have no rational points.
\newblock Arxiv Preprint, available at \url{https://arxiv.org/abs/1308.0395v1},
  2013.

\bibitem[Bou98]{Bourbaki-liegroupsliealgebras}
{N}. Bourbaki.
\newblock {\em Lie groups and {L}ie algebras. {C}hapters 1--3}.
\newblock Elements of Mathematics (Berlin). Springer-Verlag, Berlin, 1998.
\newblock Translated from the French, Reprint of the 1989 English translation.

\bibitem[BSS21]{BSSpreprint}
{M}. Bhargava, {A}. Shankar, and {A} Swaminathan.
\newblock The second moment of the size of the $2$-{S}elmer group of elliptic
  curves.
\newblock Arxiv Preprint, available at
  \url{https://arxiv.org/abs/2110.09063v1}, 2021.

\bibitem[Col85]{coleman-integrals}
{R}.~{F}. Coleman.
\newblock Torsion points on curves and {$p$}-adic abelian integrals.
\newblock {\em Ann. of Math. (2)}, 121(1):111--168, 1985.

\bibitem[CP23]{CaroPasten-chabautysurfaces}
{J}. Caro and {H}. Pasten.
\newblock A {C}habauty-{C}oleman bound for surfaces.
\newblock {\em Invent. Math.}, 234(3):1197--1250, 2023.

\bibitem[Fal83]{MR718935}
G.~Faltings.
\newblock Endlichkeitss\"{a}tze f\"{u}r abelsche {V}ariet\"{a}ten \"{u}ber
  {Z}ahlk\"{o}rpern.
\newblock {\em Invent. Math.}, 73(3):349--366, 1983.

\bibitem[Fal91]{Faltings-diophantineapproximationabelianvars}
{G}. Faltings.
\newblock Diophantine approximation on abelian varieties.
\newblock {\em Ann. of Math. (2)}, 133(3):549--576, 1991.

\bibitem[GM19]{gunther2019irrational}
{J}. Gunther and {J}.~{S}. Morrow.
\newblock Irrational points on random hyperelliptic curves.
\newblock Arxiv Preprint, available at
  \url{https://arxiv.org/abs/1709.02041v3}, 2019.

\bibitem[Gra07]{MR2340106}
{A}. Granville.
\newblock Rational and integral points on quadratic twists of a given
  hyperelliptic curve.
\newblock {\em Int. Math. Res. Not. IMRN}, (8):Art. ID 027, 24, 2007.

\bibitem[KRZB16]{MR3566201}
{E}. Katz, {J}. Rabinoff, and {D}. Zureick-Brown.
\newblock Uniform bounds for the number of rational points on curves of small
  {M}ordell-{W}eil rank.
\newblock {\em Duke Math. J.}, 165(16):3189--3240, 2016.

\bibitem[Lag22]{MR4471040}
{J}. Laga.
\newblock The average size of the 2-{S}elmer group of a family of
  non-hyperelliptic curves of genus 3.
\newblock {\em Algebra Number Theory}, 16(5):1161--1212, 2022.

\bibitem[Par16]{Par16}
{J}. Park.
\newblock Effective {C}habauty for symmetric powers of curves.
\newblock Arxiv Preprint, available at
  \url{https://arxiv.org/abs/1606.05195v1}, 2016.

\bibitem[Poo06]{MR2293593}
{B}. Poonen.
\newblock Heuristics for the {B}rauer-{M}anin obstruction for curves.
\newblock {\em Experiment. Math.}, 15(4):415--420, 2006.

\bibitem[PS14]{PoonenStoll-Mosthyperellipticnorational}
{B}. Poonen and {M}. Stoll.
\newblock Most odd degree hyperelliptic curves have only one rational point.
\newblock {\em Ann. of Math. (2)}, 180(3):1137--1166, 2014.

\bibitem[PV04]{MR2029869}
{B}. Poonen and {J}.~{F}. Voloch.
\newblock Random {D}iophantine equations.
\newblock In {\em Arithmetic of higher-dimensional algebraic varieties ({P}alo
  {A}lto, {CA}, 2002)}, volume 226 of {\em Progr. Math.}, pages 175--184.
  Birkh\"{a}user Boston, Boston, MA, 2004.
\newblock With appendices by Jean-Louis Colliot-Th\'{e}l\`ene and Nicholas M.
  Katz.

\bibitem[RT21]{MR4258170}
{B}. Romano and {J}.~{A}. Thorne.
\newblock {$E_8$} and the average size of the 3-{S}elmer group of the
  {J}acobian of a pointed genus-2 curve.
\newblock {\em Proc. Lond. Math. Soc. (3)}, 122(5):678--723, 2021.

\bibitem[Sik09]{Sik09}
{S}. Siksek.
\newblock Chabauty for symmetric powers of curves.
\newblock {\em Algebra Number Theory}, 3(2):209--236, 2009.

\bibitem[Sto09]{stoll2009average}
{M}. Stoll.
\newblock On the average number of rational points on curves of genus 2.
\newblock Arxiv Preprint, available at \url{https://arxiv.org/abs/0902.4165v1},
  2009.

\bibitem[Sto17]{Stoll-ChabautywithoutMordellWeil}
{M}. Stoll.
\newblock Chabauty without the {M}ordell-{W}eil group.
\newblock In {\em Algorithmic and experimental methods in algebra, geometry,
  and number theory}, pages 623--663. Springer, Cham, 2017.

\bibitem[Sto19]{MR3908770}
{M}. Stoll.
\newblock Uniform bounds for the number of rational points on hyperelliptic
  curves of small {M}ordell-{W}eil rank.
\newblock {\em J. Eur. Math. Soc. (JEMS)}, 21(3):923--956, 2019.

\bibitem[SW18]{MR3719247}
A.~Shankar and X.~Wang.
\newblock Rational points on hyperelliptic curves having a marked
  non-{W}eierstrass point.
\newblock {\em Compos. Math.}, 154(1):188--222, 2018.

\bibitem[VW21]{MR4309843}
{S}. Vemulapalli and {D}. Wang.
\newblock Uniform bounds for the number of rational points on symmetric squares
  of curves with low {M}ordell-{W}eil rank.
\newblock {\em Acta Arith.}, 199(4):331--359, 2021.

\bibitem[Zar96]{Zarhin-logarithmmap}
Yu.~G. Zarhin.
\newblock {$p$}-adic abelian integrals and commutative {L}ie groups.
\newblock volume~81, pages 2744--2750. 1996.
\newblock Algebraic geometry, 4.

\end{thebibliography}

\end{document}